\documentclass[12pt]{article}
\usepackage{amsmath,amsthm,amsfonts,amssymb,mathrsfs,latexsym,graphicx}
\newcommand{\la}[1]{\lambda #1}
\newcommand{\be}{\begin{eqnarray}}
\newcommand{\ee}{\end{eqnarray}}
\newcommand{\bes}{\begin{eqnarray*}}
\newcommand{\ees}{\end{eqnarray*}}

\newtheorem{theorem}{Theorem}
\newtheorem{lemma}{Lemma}
\newtheorem{corollary}{Corollary}
\newtheorem{conjecture}{Conjecture}

\begin{document}

\title
{A Simple Proof of a Conjecture of Simion \footnote{The research was
partially supported by the NSF of Liaoning Province of China.}}
\author {Yi\  Wang\\[5mm]
{\footnotesize\it Department of Applied Mathematics, Dalian
University of
Technology, Dalian  116023}\\[1mm] {\footnotesize\it People's Republic of China}\\[5mm]
E-mail: wangyi@dlut.edu.cn}
\date{Received May 16, 2002}

\maketitle

\begin{abstract}
Simion had a unimodality conjecture concerning the number of lattice
paths in a rectangular grid with the Ferrers diagram of a partition
removed. Hildebrand recently showed the stronger result that these
numbers are log concave. Here we present a simple proof of
Hildebrand's result. \vspace{2mm}

\end{abstract}

\section{Introduction}
\hspace*{\parindent} Let $\la=(\la_1,\la_2,\ldots,\la_r)$ be an
integer partition where $\la_1\ge \la_2\ge\cdots\ge\la_r\ge 0$ and
$\la'$ the conjugate of $\la$. Let $R(m,n)$ denote the rectangular
grid with $m$ rows and $n$ columns where $m\ge \la'_1$ and $n\ge
\la_1$. Consider the grid with the Ferrers diagram of $\la$ removed
from the upper left corner of $R(m,n)$. Let $N(m,n,\la)$ denote the
number of paths in $R(m,n)$ such that the path starts at the lower
left corner, the path ends at the upper right-hand corner, and at
each step the path goes up one unit or to the right one unit but
never inside the removed Ferrers diagram of $\la$. It is well known
that there would be $m+n\choose n$ such paths if there were no
Ferrers diagram removed. In \cite{Sim}, Simion proposed a
unimodality conjecture for $N(m,n,\la)$. This conjecture is also
described in \cite{B2,Sag}. The description in here is based on that
in \cite{Sag}.

\begin{conjecture}[Simion]
For each integer $\ell$ and each
partition $\la$, the sequence
$$ N(\la'_1,\la_1+\ell,\la), N(\la'_1+1,\la_1+\ell-1,\la), \ldots,
N(\la'_1+\ell,\la_1,\la)$$ is unimodal.
\end{conjecture}

A sequence of positive numbers $x_0,x_1,\ldots,x_{\ell}$ is unimodal
if $x_0\le x_1\le\ldots\le x_k\ge\ldots\ge x_{\ell}$ for some $k$
and is log concave in $i$ if $x_{i-1}x_{i+1}\le x^2_i$ for
$0<i<\ell$. It is well known that a log-concave sequence is also
unimodal. Very recently, Hildebrand~\cite{Hil} showed the following
stronger result.

\begin{theorem}[Hildebrand]\label{thm1}
The sequence in Simion's
conjecture is log concave.
\end{theorem}

The key idea behind Hildebrand's proof is to show  \be\label{eq1}
N(m,n+1,\la)N(m+1,n,\la)\le N(m,n,\la)N(m+1,n+1,\la) \ee and
\be\label{eq2} N(m-1,n+1,\la)N(m+1,n+1,\la)\le N^2(m,n+1,\la). \ee
Note that (\ref{eq1}) and (\ref{eq2}) yield \be\label{eq3}
N(m-1,n+1,\la)N(m+1,n,\la)\le N(m,n,\la)N(m,n+1,\la). \ee By
symmetry, this implies \be\label{eq4} N(m+1,n-1,\la)N(m,n+1,\la)\le
N(m,n,\la)N(m+1,n,\la). \ee Further, (\ref{eq3}) and (\ref{eq4})
yield
$$N(m+1,n-1,\la)N(m-1,n+1,\la)\le N^2(m,n,\la),$$ the desired
result. So, to show Theorem \ref{thm1}, it suffices to show
(\ref{eq1}) and (\ref{eq2}).

\section{Proof of (\ref{eq1}) and (\ref{eq2})}
\hspace*{\parindent} A matrix $A$ is said to be totally positive of
order 2 (or a $TP_2$ matrix, for short) if all the minors of order 2
of $A$ have nonnegative determinants. A sequence of positive numbers
$x_0,x_1,x_2,\ldots,x_{\ell}$ is log concave if and only if the
matrix
$$\left(
    \begin{array}{ccccc}
      x_0&x_1&x_2&\cdots&x_{\ell} \\
      &x_0&x_1&\cdots&x_{\ell-1} \\
    \end{array}
  \right)
$$ is $TP_2$ (see, e.g.,
\cite[Proposition 2.5.1]{B1}). The following lemma is a special case
of \cite[Theorem 2.2.1]{B1}.

\begin{lemma}\label{lem1}
The product of two finite $TP_2$ matrices
is also $TP_2$.
\end{lemma}

\begin{corollary}\label{cor1} Let $a_0,a_1,\ldots,a_{\ell}$ be nonnegative
and $x_0,x_1,\ldots,x_{\ell}$ positive. Denote $A_m=\sum_{i=0}^ma_i$
and $X_m=\sum_{i=0}^mx_i$ for $m=0,1,\ldots,\ell$.
\begin{enumerate}
  \item [\rm (i)] Assume $a_ix_{i+1}\le a_{i+1}x_i$ for all $i$. Then
$A_mX_{m+1}\le A_{m+1}X_m$ for all $m$.
  \item [\rm (ii)] If the sequence $x_0,x_1,\ldots,x_{\ell}$ is log
concave, then so is the sequence $X_0,X_1,\ldots,X_{\ell}$.
\end{enumerate}
\end{corollary}

\begin{proof}
Note that
$$\left(
    \begin{array}{ccccc}
      x_0&x_1&x_2&\cdots&x_{\ell} \\
      a_0&a_1&a_2&\cdots&a_{\ell} \\
    \end{array}
  \right)
\left(
  \begin{array}{ccccc}
    1&1&1&\cdots&1 \\
    &1&1&\cdots&1 \\
    &&1&\cdots&1 \\
    &&&\ddots&\vdots \\
    &&&&1 \\
  \end{array}
\right)=
\left(
  \begin{array}{ccccc}
    X_0&X_1&X_2&\cdots&X_{\ell} \\
    A_0&A_1&A_2&\cdots&A_{\ell} \\
  \end{array}
\right)
$$ and
$$
\left(
    \begin{array}{ccccc}
      x_0&x_1&x_2&\cdots&x_{\ell} \\
      &x_0&x_1&\cdots&x_{\ell-1} \\
    \end{array}
  \right)
\left(
  \begin{array}{ccccc}
    1&1&1&\cdots&1 \\
    &1&1&\cdots&1 \\
    &&1&\cdots&1 \\
    &&&\ddots&\vdots \\
    &&&&1 \\
  \end{array}
\right)= \left(
  \begin{array}{ccccc}
    X_0&X_1&X_2&\cdots&X_{\ell} \\
    0&X_0&X_1&\cdots&X_{\ell-1}. \\
  \end{array}
\right)$$ The statement follows immediately from Lemma
\ref{lem1}.\end{proof}


We now prove (\ref{eq1}) and (\ref{eq2}) by induction on $\la_1$,
the largest part of $\la$. If $\la_1=0$, i.e., $\la=\emptyset$, then
both (\ref{eq1}) and (\ref{eq2}) are easily verified since
$N(m,n,\la)={m+n\choose n}$, so we proceed to the induction step.
Let $\la_1\ge 1$ and $r=\la_1'$. Denote by $\mu$  the partition
$(\la_1-1,\ldots,\la_r-1)$.
Then
$$N(m,n,\la)=\sum_{k=\la'_1}^mN(k,n-1,\mu).$$
However, the sequence $N(k,n-1,\mu)$ is log concave in $k$ by the
induction hypothesis. Hence $N(m,n,\la)$ is log concave in $m$ by
Corollary \ref{cor1} (ii). This proves (\ref{eq2}). On the other
hand, we have by the induction hypothesis
$$N(k,n,\mu)N(k+1,n-1,\mu)\le N(k+1,n,\mu)N(k,n-1,\mu).$$
Thus by Corollary \ref{cor1} (i),
$$\sum_{k=\la'_1}^mN(k,n,\mu)\sum_{k=\la'_1}^{m+1}N(k,n-1,\mu)
\le \sum_{k=\la'_1}^mN(k,n-1,\mu)\sum_{k=\la'_1}^{m+1}N(k,n,\mu).
$$ This gives (\ref{eq1}).

\section*{Acknowledgements}
\hspace*{\parindent} This work was done when the author held a
postdoctoral fellowship at Nanjing University and supported
partially by NSF of Liaoning Province of China. The author wishes to
thank the anonymous referees for their careful reading and
corrections.

\end{document}